\documentclass[12pt]{amsart}

\usepackage{amsmath,amssymb,amsthm,color}

\usepackage[margin=1in]{geometry}

\title[Consecutive prime biases and sawtooth random variables]{The distribution of consecutive prime biases and sums of sawtooth random variables}

\author{Robert J. Lemke Oliver}
\address{Department of Mathematics, Tufts University}
\email{robert.lemke\_oliver@tufts.edu}

\author{Kannan Soundararajan}
\address{Department of Mathematics, Stanford University}
\email{ksound@stanford.edu}

\thanks{Robert Lemke Oliver was partially supported by NSF grant DMS-1601398.  Kannan Soundararajan was partially supported by NSF grant DMS-1500237 and by a Simons Investigator grant from the Simons Foundation.}

\numberwithin{equation}{section}
\newtheorem{theorem}{Theorem}
\newtheorem{lemma}[theorem]{Lemma}

\newtheorem{proposition}[theorem]{Proposition}
\numberwithin{theorem}{section}
\theoremstyle{remark} 
\theoremstyle{remark} 
\theoremstyle{remark} 

\renewcommand{\pmod}[1]{\,\left(\text{mod }#1\right)}

\begin{document}

\begin{abstract}
In recent work, we considered the frequencies of patterns of consecutive primes $\pmod{q}$ and numerically found biases toward certain patterns and against others.  We made a conjecture explaining these biases, the dominant factor in which permits an easy description but fails to distinguish many patterns that have seemingly very different frequencies.  There was a secondary factor in our conjecture accounting for this additional variation, but it was given only by a complicated expression whose distribution was not easily understood.  Here, we study this term, which proves to be connected to both the Fourier transform of classical Dedekind sums and the error term in the asymptotic formula for the sum of $\phi(n)$.
\end{abstract}

\maketitle

\section{Introduction} 

\noindent  Let $p_n$ denote the sequence of primes in ascending order.   Given $q\ge 3$ and $\mathbf{a}=(a_1,\dots,a_r)$ satisfying $(a_i,q)=1$ for all $1\leq i \leq r$, in recent work 
\cite{LOS} we studied biases in the occurrence of the pattern $\mathbf{a}$ in strings of $r$ 
consecutive primes reduced $\pmod q$.   Thus, we defined 
\[
\pi(x;q,\mathbf{a}) := \#\{p_n\leq x: p_{n+i-1} \equiv a_i \pmod{q} \text{ for } 1\leq i \leq r\},
\]
and conjectured  that   
\begin{equation} \label{eqn:bias}
\pi(x;q,\mathbf{a}) = \frac{\mathrm{li}(x)}{\phi(q)^r} \Big(1 + c_1(q;\mathbf{a}) \frac{\log\log x}{\log x} + c_2(q;\mathbf{a}) \frac{1}{\log x} + O((\log x)^{-7/4})\Big),
\end{equation}
where $c_1(q;\mathbf{a})$ and $c_2(q;\mathbf{a})$ are certain explicit constants.   The term $c_1(q;\mathbf{a})$ is easily described,  
\[
c_1(q;\mathbf{a}) = \frac{\phi(q)}{2}\Big(\frac{r-1}{\phi(q)} - \#\{i\leq r-1: a_i \equiv a_{i+1} \pmod{q}\}\Big),
\]
and it acts as a bias against immediate repetitions in the pattern $\mathbf{a}$.  The term $c_2(q;\mathbf{a})$ is more complicated, and 
the goal of this paper is to understand its distribution.   If $r\ge 3$ then 
\[
c_2(q;\mathbf{a}) = \sum_{i=1}^{r-1} c_2(q;(a_i,a_{i+1})) + \frac{\phi(q)}{2}\sum_{j=1}^{r-2} \frac{1}{j}\Big(\frac{r-1-j}{\phi(q)} - \#\{i : a_i \equiv a_{i+j+1} \pmod{q}\}\Big),
\]
so that it is sufficient to understand the case $r=2$; that is, $c_2(q;(a,b))$ with $(a,q)=(b,q)=1$.

For the sake of simplicity, we shall confine ourselves to the case when $q$ is prime.   For any character $\chi \pmod q$ we define 
\begin{equation} \label{Adef}
A_{q,\chi} = \prod_{p\nmid q} \Big(1-\frac{(1-\chi(p))^2}{(p-1)^2}\Big).
\end{equation}   
Then the quantity $c_2(q;(a,b))$ is given by 
$$ 
c_2(q;(a,a))  = \frac{q-2}{2} \log (q/2\pi), 
$$ 
and when $a\not \equiv b \pmod q$ by 
\begin{equation} \label{eqn:non-diag}
c_2(q;(a,b)) = \frac{1}{2} \log \frac{2\pi}{q} + \frac{q}{\phi(q)} \sum_{\chi \neq \chi_0\pmod{q}} \Big( \overline{\chi}(b-a) +
\frac{1}{\phi(q)} (\overline{\chi}(b)-\overline{\chi}(a)) \Big) L(0,\chi)L(1,\chi) A_{q,\chi}. 
\end{equation}
The diagonal term $c_2(q;(a,a))$ is thus completely explicit, and of size $q\log q$.   Our work here shows that the off-diagonal terms $c_2(q;(a,b))$ can also be large; usually they are of size about $q$, occasionally getting to size $q\log \log q$ (attaining both positive and negative values),  which we believe is their maximal size.

Before stating our result, we make one more simplification.   Define 
\begin{equation} 
\label{Cdef} 
C(k) =C(k;q)= \frac{1}{\phi(q)} \sum_{\chi\neq \chi_0 \pmod q} \overline{\chi(k)} L(0,\chi)L(1,\chi) A_{q,\chi}. 
\end{equation}  
Since $A_{q,\chi} \ll 1$, $L(1,\chi) \ll \log q$, and (upon using the functional equation) $L(0,\chi) \ll \sqrt{q}\log q$, from \eqref{eqn:non-diag} it follows that for $a\not\equiv b\pmod q$ 
\begin{equation} 
\label{1.5}
\frac{c_2(q;(a,b))}{q} = C(b-a) + O\Big( \frac{(\log q)^2}{\sqrt{q}} \Big). 
\end{equation} 
Thus for large $q$ it is enough to understand the distribution of $C(k)$ as $k$ varies over all non-zero residue classes 
$\pmod q$.  Since $L(0,\chi) = 0$ for even characters $\chi$, in \eqref{Cdef} only odd characters $\chi$ make a contribution, and therefore $C(k) = - C(-k)$ is an odd function of $k$.  

\begin{theorem} \label{thm1}   (1) As $q\to \infty$ the distribution of $C(k)$ tends to a continuous probability distribution, symmetric 
around $0$.  Precisely, there is a continuous function $\Phi_C$ with $\Phi_C(-x)+ \Phi_C(x) =1$ such that uniformly for all $x \in [-X,X]$ 
one has 
$$ 
\frac{1}{q} \# \{ k\pmod q: \ C(k) \le \tfrac {e^{\gamma}}{2} x \} = \Phi_C(x) + o(1). 
$$ 

(2)  Uniformly for all $e \le x\le (\frac 12-\epsilon) \log \log q$ one has 
$$ 
\exp( -A_1 e^x/x) \,\,\,\ge\,\,\, \frac{1}{q} \# \{ k \pmod q: \ C(k) \ge \tfrac{e^{\gamma}}{2} x \} \,\,\,\ge\,\,\, \exp(- A_2  e^{x}\log x) 
$$ 
for some positive constants $A_1$ and $A_2$.  

(3)   For all large $q$, there exists $k \pmod q$ with 
$$ 
-C(-k) = C(k) \ge \Big( \frac{e^{\gamma}}{4}  -\epsilon\Big) \log \log q. 
$$ 

(4)  For all $k \pmod q$ we have 
$$ 
C(k) \ll (\log q)^{\frac 23} (\log \log q)^2. 
$$ 

(5)  The values $C(k)$ have an ``almost periodic" structure.   Precisely, suppose $1\le m \le q/4$ is 
a multiple of every natural number below $B \ge 2$.  Then 
$$ 
\frac 1q \sum_{k \pmod q} |C(k) - C(k+m)|^2 \ll \frac{1}{B^{1-\epsilon}} + \frac{m}{q} \log B. 
$$
\end{theorem}  

We make a few comments concerning Theorem \ref{thm1} before proceeding to related results.  
In part 1, we believe that the distribution for $C(k)$ has a density, which is to say that $\Phi_C$ 
is in fact differentiable.  Our proof falls just a little short of establishing this.   In part 2, there is a 
gap between the upper and lower bounds for the tail frequencies.  With a little more care, we 
can improve the lower bound there to $\exp(-A_3 e^x)$ for a suitable positive constant $A_3$, but 
there still remains a gap between the two bounds.   The distribution of $C(k)$, and especially the 
double exponential decay seen in part 2, are reminiscent of the distribution of values of $L(1,\chi_d)$ (see \cite{GS}).  
Motivated by this analogy, or by extrapolating the lower bounds in part 2, we believe that in part 3 there should exist values of $C(k)$ as 
large as $(\frac{e^{\gamma}}{2} -\epsilon) \log \log q$.  We also conjecture that $(\frac{e^{\gamma}}{2} +\epsilon) \log \log q$ 
should be the largest possible value of $C(k)$, which would be a substantial strengthening of part 4.  Finally, in addition 
to the almost periodic structure given in part 5 (where $k$ varies), there should be an almost periodic structure as $q$ varies.  That is, 
if $q_1$ and $q_2$ are two large random primes with $q_1-q_2$ being a multiple of the numbers below $B$, then $C(k;q_1)$ and $C(k;q_2)$ 
will be close to each other (on average over $k$).   We hope that an interested reader will embrace some of 
these remaining problems.

While the quantity $C(k)$ is the main focus of this paper, closely related objects arise in two other seemingly unrelated 
contexts.   The first of these concerns Dedekind sums. 
Let $\psi(x)$ denote the 
sawtooth function defined by 
\[ 
\psi(x) = \begin{cases} 
\{ x \} -1/2 &\text{ if  } x \not \in \mathbf{Z}, \\ 
0 &\text{ if } x \in \mathbf{Z},
\end{cases}
\] 
which is an odd function, periodic with period $1$.   If $q$ is prime and $a$ is a reduced residue $\pmod{q}$, then the Dedekind sum $s_q(a)$ is defined by
\begin{equation}\label{eqn:dedekind-def}
s_q(a) := \sum_{x\pmod{q}} \psi\Big(\frac{x}{q}\Big) \psi\Big(\frac{ax}{q}\Big).
\end{equation}
The Dedekind sum arises naturally in number theory when studying the modular transformation properties of the Dedekind $\eta$-function, but it also appears in other contexts and satisfies many interesting properties \cite{Apostol,Vardi}.   We study here 
the discrete Fourier transform of the Dedekind sum $s_q(a)$.  Thus for a prime $q$ and residue class $t\pmod q$ we define 
\begin{equation} \label{eqn:fourier-transform-def}
\widehat{s}_q(t) := \frac{1}{q} \sum_{a\pmod{q}} s_q(a) e(at/q), 
\end{equation}
where $e(z) =e^{2\pi iz}$ throughout.   In Lemma \ref{lem:dedekind-sum}  we shall see that 
\[
\widehat{s}_q(t) = \frac{-1}{\pi i \phi(q)} \sum_{\chi\neq \chi_0\pmod{q}} \bar\chi(t) L(0,\chi) L(1,\chi),
\]
so that $\widehat{s}_q(t)$ is indeed a simpler version of $C(k)$.   An alternative useful expression is 
\begin{equation}\label{eqn:dedekind-sawtooth}
\widehat{s}_q(t) = \frac{1}{\pi i} \sum_{\substack{n=1 \\ (n,q)=1}}^{\infty} \frac{\psi(t\overline{n}/q)}{n}, 
\end{equation} 
where $\overline n$ denotes the multiplicative inverse of the reduced residue class $n\pmod q$ and the 
sum converges since the partial sums $\sum_{n\le x, (n,q)=1} \psi(t\overline{n}/q)$ are bounded.   

\begin{theorem}
\label{thm2} 
(1) As $q\to \infty$ the distribution of $\pi i \widehat{s}_q(t)$ tends to a continuous probability distribution, symmetric 
around $0$.  Precisely, there is a continuous function $\Phi_{s}$ with $\Phi_{s}(-x)+ \Phi_{s}(x) =1$ such that uniformly for all $x \in [-X,X]$ 
one has 
$$ 
\frac{1}{q} \# \{ t\pmod q: \ \pi i \widehat{s}_q(t) \le \tfrac {e^{\gamma}}{2} x \} = \Phi_{ s}(x) + o(1). 
$$ 

(2)  Uniformly for all $e \le x\le (\frac 12-\epsilon) \log \log q$ one has 
$$ 
\exp( -A_1 e^x/x) \,\,\, \ge\,\,\, \frac{1}{q} \# \{ t \pmod q: \  \pi i \widehat{s}_q(t) \ge \frac{e^{\gamma}}{2} x \} \,\,\,\ge\,\,\, \exp(- A_2  e^{x}\log x) 
$$ 
for some positive constants $A_1$ and $A_2$.  

(3)   For all large $q$, there exists $t \pmod q$ with 
$$ 
-\pi i \widehat{s}_q(-t) =\pi i  \widehat{s}_q(t) \ge \Big( \frac{e^{\gamma}}{4}  -\epsilon\Big) \log \log q. 
$$ 

(4)  For all $t \pmod q$ we have 
$$ 
\widehat{s}_q(t) \ll (\log q)^{\frac 23} (\log \log q)^2. 
$$ 

(5)  The values $\widehat{s}_q(t)$ have an ``almost periodic" structure.   Precisely, suppose $1\le m \le q/4$ is 
a multiple of every natural number below $B \ge 2$.  Then 
$$ 
\frac 1q \sum_{t \pmod q} |\widehat{s}_q(t) - \widehat{s}_q(k+m)|^2 \ll \frac{1}{B^{1-\epsilon}} + \frac{m}{q} \log B. 
$$
\end{theorem} 

Theorem \ref{thm2} exactly parallels the results of Theorem \ref{thm1}, with the same deficiencies discussed 
there.   The proofs of Theorems \ref{thm1} and \ref{thm2} are nearly identical, and so we give details only for 
Theorem \ref{thm1}.  

Our third topic concerns the remainder term in the asymptotic for the mean value of Euler's $\phi$-function.  
Define the quantity $R(x)$ by the relation
\[
\sum_{n\leq x} \phi(n) = \frac{3}{\pi^2} x^2 + R(x).
\]
Simple arguments show that $R(x) \ll x \log x$, and Walfisz \cite{Walfisz} established that  $R(x) \ll x (\log x)^{2/3} (\log\log x)^{4/3}$, which is presently the best known estimate.  Montgomery \cite{Montgomery} conjectured that $R(x) \ll x \log\log x$ and $R(x) = \Omega_{\pm}(x \log\log x)$, and he showed that  $R(x) = \Omega_\pm(x \sqrt{\log\log x})$.  Key to Montgomery's work is the expression
\[
R(x) = \frac{\phi(x)}{2} - x \sum_{n\leq x} \frac{\mu(n) \psi(x/n)}{n} + O\Big(x \exp(-c\sqrt{\log x})\Big)
\]
for some positive constant $c$, where $\phi(x) = 0$ if $x\not\in\mathbf{Z}$.  The sum in this expression is akin to the equation \eqref{eqn:dedekind-sawtooth} for $\widehat{s}_q(t)$ 
with $\overline{n}/q$ replaced by $1/n$ and with the weight $1/n$ replaced with $\mu(n)/n$.  
Accordingly, many of the techniques used to prove Theorems \ref{thm1} and \ref{thm2} apply to $R(x)$ as well, though unfortunately with less precision owing to 
the presence of $\mu(n)$.  For convenience, we define $\widetilde{R}(x) = R(x)/x - \phi(x)/2x$.

\begin{theorem}\label{thm3}
As $y\to \infty$ the distribution of $\widetilde{R}(u)$ for real $u\leq y$ tends to a probability distribution, symmetric 
around $0$.  Precisely, there is a function $\Phi_R$ with $\Phi_R(-x)+ \Phi_R(x) =1$ such that uniformly for all $x \in [-X,X]$ 
one has 
$$ 
\frac{1}{y} \mathrm{meas}(\{ u \leq y: \ \widetilde{R}(u) \le \tfrac {3e^{\gamma}}{\pi^2} x \}) = \Phi_R(x) + o(1), 
$$ 
where $\mathrm{meas}(I)$ denotes the Lebesgue measure of $I \subseteq \mathbf{R}$.
Moreover, uniformly for all $e \le x\le (\frac 12-\epsilon) \log \log y$ one has 
$$
\frac{1}{y} \mathrm{meas}(\{ u \leq y: \ \widetilde{R}(u) \ge \tfrac {3e^{\gamma}}{\pi^2} x \}) \leq \exp(-A_1 e^x/x)
$$
for some positive constant $A_1$.
\end{theorem}

We prove Theorem \ref{thm3} by showing that all positive integral moments of 
$\widetilde{R}(n)$ exist and are not too large.   The moment calculation  
refines earlier work of Pillai and Chowla \cite{PillaiChowla} and Chowla \cite{Chowla}, who computed the mean and variance respectively: 
\[
\sum_{n\leq x} \widetilde{R}(n) = o(x) \quad \text{and} \quad \frac{1}{y}\int_0^y \widetilde{R}(u)^2 \,du \sim \frac{1}{2\pi^2}.
\]
In Theorem \ref{thm3}, using Montgomery's construction in his $\Omega$-result, we can obtain a lower bound 
for the frequency of large values of ${\widetilde R}(u)$ of the form $\exp(-e^{x^{2+\epsilon}})$, which is very far from the 
upper bound.  We expect that there is a lower bound similar to that in Theorems \ref{thm1} and \ref{thm2} in this situation also, 
and this would be in keeping with Montgomery's conjecture on the true size of ${\widetilde R}(u)$.  

\subsection*{Organization}  Our main focus is the proofs of Theorems \ref{thm1} and \ref{thm2}.  We establish preliminary results useful for both in Sections \ref{sec:prelim} and \ref{sec:B}.  We then prove Theorem \ref{thm1} in Sections \ref{sec:moments}-\ref{sec:proofthm}; since the proof of Theorem \ref{thm2} follows along identical lines,  we omit it.  In Section \ref{sec:montgomery}, we discuss the modifications that lead to Theorem \ref{thm3}.

\section{First steps} \label{sec:prelim}

Here we establish some formulae for $\widehat{s}_q(t)$ and $C(k)$ which will be 
the basis for our subsequent work. 

\begin{lemma} \label{lem:dedekind-sum}
Let $q$ be prime.  For any $(t,q)=1$, we have
\[
\widehat{s}_q(t) = \frac{-1}{\pi i \phi(q) }\sum_{\chi\neq \chi_0\pmod{q}} \bar\chi(t) L(0,\chi) L(1,\chi) = \frac{1}{\pi i} \sum_{\substack{n=1\\(n,q)=1}}^{\infty} \frac{\psi(t\overline{n}/q)}{n}.
\]
Moreover, for any $x\ge 1$ we have 
$$ 
\widehat{s}_q(t) = \frac{1}{\pi i} \sum_{\substack{n\le x \\ (n,q)=1}} \frac{\psi(t\overline{n}/q)}{n} + O\Big(\frac qx \Big). 
$$ 
\end{lemma}

\begin{proof}
For any non-principal character $\chi \pmod{q}$, we have (see, e.g., \cite[Theorem 4.2]{Washington})
\begin{equation} \label{eqn:l-zero}
L(0,\chi) = -\sum_{a\pmod{q}} \chi(a) \psi(a/q).
\end{equation}
Notice that $L(0,\chi)=0$ if $\chi$ is an even character, and that right side of  the formula in \eqref{eqn:l-zero} evaluates to 0 if $\chi$ is principal.  The functional equation for odd characters gives
\[
L(1,\chi) = -\frac{\tau(\chi)\pi i}{q} L(0,\bar\chi),  
\]
where  $\tau(\chi) = \sum_{m\pmod{q}} \chi(m) e(m/q)$ denotes the Gauss sum.   
Thus we obtain
\begin{align*}
\sum_{\chi\neq\chi_0\pmod{q}} \bar\chi(t)L(0,\chi)L(1,\chi) 
	&= -\frac{\pi i}{q} \sum_{\chi\pmod{q}} \tau(\chi) \Big|\sum_{a\pmod{q}} \chi(a)\psi\Big(\frac{a}{q}\Big)\Big|^2 \\
	&= -\frac{\pi i}{q} \sum_{a,b,m \pmod{q}} e(m/q)\psi\Big(\frac{a}{q}\Big)\psi\Big(\frac{b}{q}\Big)\sum_{\chi\pmod{q}} \chi(am)\bar\chi(bt) \\
	&= -\frac{\phi(q) \pi i}{q} \sum_{a,b \not\equiv 0 \pmod{q}} e\Big(\frac{tb\overline{a}}{q}\Big)\psi\Big(\frac{a}{q}\Big)\psi\Big(\frac{b}{q}\Big) \\
	&= -\phi(q) \pi i\, \widehat{s}_q(t).
\end{align*}
The first identity in the lemma follows.

To obtain the second identity, note that using \eqref{eqn:l-zero} and the orthogonality relation 
for characters  
\begin{align} 
\label{2.2}
-\sum_{\chi \neq\chi_0 \pmod q} \overline{\chi}(t) L(0,\chi) \sum_{n\le N}\frac{\chi(n)}{n} 
&= \sum_{\chi \pmod q} \overline{\chi}(t) \sum_{a\pmod q}  \chi(a) \psi(a/q) \sum_{n\le N} \frac{\chi(n)}{n} \nonumber \\
&= \phi(q) \sum_{\substack{n\le N \\ (n,q)=1}}  \frac{1}{n} \psi(t\overline{n}/q). 
\end{align} 
 Letting $N \to \infty$, the second identity follows. 
 
 To obtain the truncated version, note that 
 $$ 
 \Big| \sum_{\substack{n\le x \\ (n,q)=1} } \psi(t\overline{n}/q) \Big| \le q 
 $$ 
 trivially, and therefore 
 $$ 
 \sum_{\substack{ n> x \\ (n,q)=1} } \frac{\psi(t\overline{n}/q)}{n} = \int_x^{\infty} \frac{1}{y^2} \sum_{x<n\le y} \psi(t\overline{n}/q) dy 
 \ll \frac{q}{x}. 
 $$ 
\end{proof}

Recall the definition of $A_{q,\chi}$ from \eqref{Adef}.   Expanding this product out, we find 
\begin{equation} 
\label{Adef2} 
A_{q,\chi} 
	= (2\chi(2)-\chi(2)^2) \prod_{p\nmid 2q} \Big(1-\frac{1}{(p-1)^2}\Big) \Big(1 + \frac{2\chi(p) - \chi(p)^2}{p^2-2p}\Big) 
	= C \sum_{n=1}^\infty a(n) \chi(2n). 
\end{equation} 
Here  
\begin{equation} 
\label{Adef3} 
C = 2 \prod_{\substack{p\ge 3 \\ p \nmid q} } \Big(1- \frac{1}{(p-1)^2} \Big), 
\end{equation} 
and $a(n)$ is a multiplicative function defined by  $a(2)=-1/2$ and $a(2^v) =0$ for all $v\ge 2$, and for 
odd primes $p$ we have 
\begin{equation} 
\label{Adef4} 
a(p) = \frac{2}{p(p-2)}, \qquad a(p^2) = -\frac{1}{p(p-2)}, \qquad \text{and } \qquad a(p^v)= 0 \text{  for all } v \ge 3. 
\end{equation} 
From the definition of $a(n)$ it is easy to check that $\sum_{n=1}^{\infty} |a(n)|n^{\sigma}$ converges for all $\sigma <1/2$ so that 
\begin{equation} 
\label{Adef5} 
\sum_{n\ge N} |a(n)| \ll N^{-\frac 12+\epsilon} \qquad \text{and } \qquad C \sum_{\substack{n\le N\\ (n,q)=1}} a(n) = 1 +O(N^{-\frac 12+\epsilon}). 
\end{equation} 

\begin{lemma}\label{lem2.2}  Define the multiplicative function $b(n)$ by setting $b(n) = \sum_{uv=n} a(u)/v$, so that $b(n)=0$ 
unless $n$ is odd and square-free, and $b(p) = 1/(p-2)$ for all odd primes $p$.  Then for any natural number $N$ we 
have 
$$ 
C(k) = - C \sum_{\substack{ n\le N \\ (n,q)=1}} b(n) \psi(k\overline{2n}/q) + O( q^{\frac 32+\epsilon} N^{-\frac 14+\epsilon}). 
$$ 
\end{lemma} 
\begin{proof}   Arguing as in \eqref{2.2} we find 
\begin{equation} 
\label{2.65} 
\frac{1}{\phi(q)} \sum_{\chi \neq \chi_0 \pmod q} \overline{\chi(k)} L(0,\chi) \sum_{\substack{n\le N \\ (n,q)=1}}  b(n) \chi(2n) = 
-\sum_{\substack{n\le N \\ (n,q)=1}} b(n) \psi(k \overline{2n}/q). 
\end{equation} 
Now if $n =uv \le N$ then either $u\le \sqrt{N}$ or $v\le \sqrt{N}$ and $\sqrt{N} < u\le N/v$.   Therefore 
\begin{equation} 
\label{2.7}
\sum_{n\le N} b(n) \chi(2n) = \sum_{u\le \sqrt{N}} a(u) \chi(2u) \sum_{v\le N/u} \frac{\chi(v)}{v} + \sum_{v\le \sqrt{N}} \frac{\chi(v)}{v} \sum_{\sqrt{N} < u \le N/v} a(u) \chi(2u). 
\end{equation} 
Bounding the partial sums of characters trivially, we find 
\begin{equation} 
\label{2.9}
L(1,\chi) = \sum_{ n\le x} \frac{\chi(n)}{n}  + \int_{x}^{\infty} \sum_{x<n\le y} \chi(n) \frac{dy}{y^2} = \sum_{n\le x} \frac{\chi(n)}{n} + O\Big( \frac{q}{x}\Big),
\end{equation} 
and so the first term in \eqref{2.7} is (using \eqref{Adef5}) 
\begin{align*}
&\sum_{u\le \sqrt{N}} a(u) \chi(2u) \Big( L(1,\chi) + O\Big( \frac{qu}{N}\Big) \Big)\\
 =& 
C^{-1} A_{q,\chi}L(1,\chi) + O( (\log q)N^{-\frac 14+\epsilon}) + O(qN^{-\frac 12})\\
=& C^{-1} A_{q,\chi} L(1,\chi)+ O(qN^{-\frac 14+\epsilon}). 
\end{align*}

As for the second term in \eqref{2.7}, using \eqref{Adef5} we may bound this by 
$$ 
\ll \sum_{v\le \sqrt{N}} \frac{1}{v}  N^{-\frac 14+\epsilon}  \ll N^{-\frac 14+\epsilon}. 
$$ 
We conclude that 
$$ 
C\sum_{\substack{n\le N \\ (n,q)=1}} b(n) \psi(k \overline{2n}/q) = - \frac{1}{\phi(q)} \sum_{\chi\neq \chi_0 \pmod q} \overline{\chi(k)} 
L(0,\chi) \Big( A_{q,\chi} L(1,\chi) + O(qN^{-\frac 14+\epsilon})\Big),
$$ 
and since $L(0,\chi) \ll \sqrt{q} \log q$, the lemma follows. 
\end{proof}

Lemmas \ref{lem:dedekind-sum} and \ref{lem2.2} give crude approximations to $\widehat{s}_q(t)$ and $C(k)$ by 
long sums (for example taking $x=q^2$ in Lemma \ref{lem:dedekind-sum}, or taking $N= q^8$ in Lemma \ref{lem2.2}).   
However, on average over $t$ or $k$, it is possible to approximate these quantities by very short sums.    

\begin{lemma} \label{lem2.3}   Let $1\le B< q$ be a real number.   Then 
$$ 
\frac{1}{\phi(q)} \sum_{k \pmod q} \Big| C(k) + C\sum_{n\le B} b(n) \psi(k\overline{2n}/q) \Big|^2 \ll B^{-1+\epsilon}  , 
$$ 
and 
$$ 
\frac{1}{\phi(q)} \sum_{t \pmod q} \Big| \widehat{s}_q(t) - \frac{1}{\pi i} \sum_{n\le B} \frac{\psi(t\overline{n}/q)}{n} \Big|^2 \ll B^{-1+\epsilon} . 
$$ 
\end{lemma} 
\begin{proof}   We shall content ourselves with proving the estimate for $C(k)$, the situation for ${\widehat{s}}_q(t)$ being 
entirely similar.   Using \eqref{2.65} and Lemma \ref{lem2.2} we see that 
\begin{align*} 
&\frac{1}{\phi(q)} \sum_{k \pmod q} \Big| C(k) + C\sum_{n\le B} b(n) \psi(k\overline{2n}/q) \Big|^2 \\
=& 
\frac{1}{\phi(q)} \sum_{k\pmod q} \Big| \frac{C}{\phi(q)} \sum_{\chi \neq \chi_0 \pmod q} \overline{\chi(k)} L(0,\chi) \sum_{\substack{ 
B< n \le q^{10} }} b(n) \chi(2n) + O(q^{-1+\epsilon}) \Big|^2.
\end{align*}
Using the orthogonality of characters to evaluate the sum over $k$, this is 
$$
\ll \frac{1}{\phi(q)^2} \sum_{\chi \neq \chi_0 \pmod q} |L(0,\chi)|^2 \Big| \sum_{B < n \le q^{10}} b(n) \chi(n) \Big|^2  + q^{-2+\epsilon}, 
$$ 
and using \eqref{2.9} and the functional equation this is 
$$  
\ll \frac{1}{q} \sum_{\chi \neq \chi_0 \pmod q}   \Big| \sum_{m\le q^2} \frac{\chi(m)}{m} \sum_{B< n \le q^{10}} b(n) \chi(n)\Big|^2 + q^{-2+\epsilon}. 
$$ 

Write temporarily 
$$ 
 \sum_{m\le q^2} \frac{\chi(m)}{m} \sum_{B< n \le q^{10}} b(n) \chi(n) = \sum_{B< n \le q^{12} } \frac{\alpha(n)}{n} \chi(n), 
$$ 
for some coefficients $\alpha(n) \ll n^{\epsilon}$.   Then (including also the contribution of $\chi_0$ below) 
$$ 
\frac{1}{q} \sum_{\chi \neq \chi_0 \pmod q} \Big|\sum_{B < n \le q^{12}} \frac{\alpha(n)}{n} \chi(n) \Big|^2 
\ll \sum_{\substack{ B < n_1, n_2 \le q^{12} \\ n_1 \equiv n_2 \pmod q}} \frac{|\alpha(n_1) \alpha(n_2)|}{n_1 n_2}.
$$ 
The terms with $n_1$, $n_2$ both below $q$ (so that $n_1 =n_2$) contribute 
$$ 
\ll \sum_{B <n <q} \frac{n^{\epsilon}}{n^2} \ll B^{-1+\epsilon}.
$$ 
The terms with $\max (n_1, n_2) \ge q$ contribute (assume without loss of generality that $n_2$ is the larger one) 
$$ 
\ll q^{\epsilon} \sum_{B< n_1 \le q^{12}} \frac{1}{n_1} \sum_{\substack{ q < n_2 \le q^{12}  \\ n_2 \equiv n_1\pmod q}} \frac{1}{n_2} 
\ll q^{\epsilon} \log q \frac{\log q}{q} \ll q^{-1+\epsilon}.
$$ 
Assembling these estimates, the lemma follows. 
\end{proof}  


\section{A key quantity}  \label{sec:B}

We shall study ${\widehat s}_q(t)$ and $C(k)$ by computing their moments, and the following key quantity will arise 
in this context.  Let $\ell$ be a natural number, and suppose $n_1$, $\ldots$, $n_\ell$ are $\ell$ natural numbers.  Then 
set 
\begin{equation} 
\label{Bdef} 
{\mathcal B}(n_1,\ldots, n_\ell) = \frac{1}{n_1\cdots n_\ell} \int_0^{n_1\cdots n_\ell} \prod_{j=1}^{\ell} \psi(x/n_j) dx. 
\end{equation}

\begin{proposition} \label{propB}  The quantity $\mathcal{B}(n_1,\dots,n_\ell)$ satisfies the following properties. 

(1)  If $\ell$ is odd then ${\mathcal B}(n_1,\ldots, n_\ell)=0$.  For even $\ell$ we have 
\begin{equation} 
\label{3.2} 
{\mathcal B}(n_1,\ldots, n_\ell) = \Big( \frac{i}{2\pi}\Big)^{\ell} \sum_{\substack{k_1, \ldots , k_\ell \neq 0\\ \sum k_j/n_j =0 }} \frac{1}{k_1\cdots k_\ell},  
\end{equation} 
where the sum is over all non-zero integers $k_j$, and this sum is absolutely convergent.  In the case $\ell=2$ one has 
$$ 
{\mathcal B}(n_1,n_2) = \frac{(n_1,n_2)^2}{12 n_1n_2}.
$$ 

(2) If $p$ is a prime dividing $n_j$ and such that $p$ does not divide any other $n_i$, then 
$$ 
{\mathcal B}(n_1,\ldots,n_j,\ldots,n_\ell) = \frac{1}{p} {\mathcal B}(n_1,\ldots, n_j/p, \ldots, n_\ell). 
$$ 

(3) If we write $n_1\cdots n_\ell = rs$ where $r$ and $s$ are coprime and $r$ is square-free while $s$ is square-full then 
$$
|{\mathcal B}(n_1,\ldots, n_\ell)| \le 2^{-\ell}r^{-1}.
$$  
\end{proposition} 


 We begin by recalling the Fourier expansion of the sawtooth function.  
Note that ${\widehat \psi}(0)=0$ and for $k\neq 0$ we have  
\begin{equation} 
\label{Fex1} 
{\widehat \psi}(k) = \int_0^1 \psi(x) e(-kx) dx = \frac{1}{-2\pi i k} = \frac{i}{2\pi k},
\end{equation}  
and so
\begin{equation} \label{eqn:sawtooth-fourier}
\psi(x) =  i \sum_{k\neq 0} \frac{e(kx)}{2\pi k}.
\end{equation}
This series converges conditionally pointwise for each $x\not\in \mathbf{Z}$, and also in the $L^2$-sense.  
For any non-negative integer $N$, recall also the Fejer kernel 
\begin{equation} 
\label{Fejer1} 
K_N(x) = \sum_{j=-N}^{N} \Big( 1- \frac{|j|}{N+1} \Big) e(jx) = \frac{1}{N+1} \Big( \frac{\sin (\pi (N+1)x)}{\sin \pi x} \Big)^2.  
\end{equation}  
We shall find it convenient to replace $\psi(x)$ by the approximation $\psi_N(x)$ defined by  
\begin{equation} 
\label{Fex3} 
\psi_N(x) = i \sum_{0<|k|\le N} \frac{e(kx)}{2\pi k}\Big(1-\frac{|k|}{N+1}\Big).  
\end{equation} 
Note that $\psi_N$ is the convolution of $\psi$ with the Fejer kernel $K_N$ 
$$
\psi_N(x) = \int_{0}^{1} \psi(y) K_N(x-y) dy, 
$$ 
and so 
\begin{equation} 
\label{Fex4} 
|\psi_N(x) - \psi(x)| \ll \min \Big( 1, \frac{1}{N \Vert x \Vert}\Big), 
\end{equation} 
which implies that 
\begin{equation} 
\label{Fex5} 
\int_0^1 | \psi_N(x) - \psi(x)| dx  \ll \frac{1+ \log N}{N}. 
\end{equation} 
Note also that $|\psi_N(x)| \le 1/2$ always.

\begin{proof}[Proof of Proposition \ref{propB}: Part 1]  Since $\psi$ is an odd function, it is clear that ${\mathcal B}(n_1,\ldots, n_\ell) =0$ 
for odd $\ell$.  Now suppose $\ell$ is even.    By Parseval it follows that 
\begin{equation} 
\label{Pars1}  
\frac{1}{n_1\cdots n_\ell} \int_0^{n_1\cdots n_\ell} \psi_N(x/n_1)\cdots \psi_N(x/n_\ell) dx = \Big(\frac{i}{2\pi}\Big)^{\ell} 
\sum_{\substack{0<|k_j| \le N \\ \sum k_j/n_j =0 }} 
\frac{1}{k_1\cdots k_\ell} \prod_{j=1}^{\ell} \Big(1- \frac{|k_j|}{N+1}\Big). 
\end{equation} 
For any complex numbers $\alpha_1$, $\ldots$, $\alpha_\ell$ and $\beta_1$, $\ldots$, $\beta_\ell$ note  
the simple identity 
\begin{equation} 
\label{identity} 
\alpha_1\cdots \alpha_\ell - \beta_1 \cdots \beta_\ell  = (\alpha_1 -\beta_1)\alpha_2\cdots \alpha_\ell+ \beta_1 (\alpha_2-\beta_2) \alpha_3\cdots \alpha_\ell + \beta_1\cdots \beta_{\ell-1} (\alpha_\ell -\beta_\ell). 
\end{equation} 
Applying this, we obtain 
$$ 
|\psi(x/n_1) \cdots \psi(x/n_\ell) - \psi_{N}(x/n_1)\cdots \psi_{N}(x/n_{\ell})| \le \frac{1}{2^{\ell -1}} \sum_{j=1}^{\ell} |\psi(x/n_j) -
\psi_{N}(x/n_j)|,  
$$ 
and so by \eqref{Pars1} and \eqref{Fex5} we conclude that 
\begin{align} 
\label{Pars2} 
{\mathcal B}(n_1,\ldots,n_\ell) &=\frac{1}{n_1\cdots n_\ell} \int_0^{n_1\cdots n_\ell}  \psi(x/n_1)\cdots \psi(x/n_\ell) dx \nonumber \\
&=  \Big(\frac{i}{2\pi}\Big)^{\ell} \sum_{\substack{0<|k_j| \le N \\ \sum k_j/n_j =0 }} 
\frac{1}{k_1\cdots k_\ell} \prod_{j=1}^{\ell} \Big(1- \frac{|k_j|}{N+1}\Big) + O\Big( \frac{1+\log N}{N}\Big). 
\end{align} 

We now show that 
$$ 
\sum_{\substack{0<|k_j| \le N \\ \sum k_j/n_j =0 }} \frac{1}{|k_1\cdots k_\ell|}
$$ 
is bounded, so that \eqref{Pars2} will imply (letting $N\to \infty$) the stated formula \eqref{3.2} for ${\mathcal B}(n_1,\ldots,n_\ell)$ 
and that the sum there converges absolutely.    By Parseval 
$$ 
\sum_{\substack{0<|k_j| \le N \\ \sum k_j/n_j =0 }} \frac{1}{|k_1\cdots k_\ell|}
= \frac{1}{n_1\cdots n_{\ell}} \int_0^{n_1 \cdots n_\ell} \prod_{j=1}^{\ell} \Big( \sum_{0< |k_j| \le N} \frac{e(k_j x/n_j)}{|k_j|} \Big) dx.
$$ 
One may check that (with $\Vert x\Vert$ denoting the distance of $x$ from the nearest integer) 
$$ 
\sum_{0< |k| \le N} \frac{e(k\theta)}{|k|} \ll \log \min \Big( N, \frac{1}{\Vert \theta \Vert}\Big) \ll \log \frac{N}{1+N\Vert \theta\Vert}. 
$$ 
 Using this and the arithmetic-geometric mean inequality above, we find 
 \begin{align*}
\sum_{\substack{0<|k_j| \le N \\ \sum k_j/n_j =0 }} \frac{1}{|k_1\cdots k_\ell|} &\ll 
\sum_{j=1}^{\ell} \frac{1}{n_1\cdots n_\ell} \int_0^{n_1 \cdots n_\ell} \Big( \log \frac{N}{1+N \Vert x/n_j\Vert} \Big)^{\ell} dx 
\\
&\ll \int_0^1 \Big( \log \frac{N}{1+N\Vert x\Vert} \Big)^{\ell} dx \ll 1.
\end{align*} 
This proves our claim, and establishes \eqref{3.2}.  

If $\ell =2$ then the condition $k_1/n_1+k_2/n_2 =0$ means that $k_1= r n_1/(n_1,n_2)$ and $k_2 = -r n_2/(n_1,n_2)$ 
for some non-zero integer $r$.  Therefore 
$$ 
{\mathcal B}(n_1, n_2) = -\frac{1}{4\pi^2} \sum_{r\neq 0} \frac{-1}{r^2} \frac{(n_1,n_2)^2}{n_1n_2} = \frac{(n_1,n_2)^2}{12 n_1n_2}. 
$$ 
\end{proof}

\begin{proof}[Proof of Proposition \ref{propB}: Parts 2 and 3]  If $p$ divides $n_j$ and no other $n_i$, then, in \eqref{3.2}, 
$k_j$ must necessarily be a multiple of $p$.  Cancelling $p$ from $k_j$ and $n_j$, Part 2 follows.  Part 3 follows from Part 2, and noting that $|{\mathcal B}(n_1,\ldots,n_\ell)| \le 2^{-\ell}$ always.  
\end{proof} 

For computing the moments of ${\widehat{s}}_q(t)$ and $C(k)$  the following proposition, which connects correlations of the sawtooth
function with ${\mathcal B}$, will be very useful. 

\begin{proposition} \label{propdiscrete}  Let $n_1$, $\ldots$, $n_\ell$ be positive integers.  Define 
$K=n_1\cdots n_{\ell}/\min(n_1, \ldots, n_{\ell})$.  If $K< q/\ell$ then  
$$
\frac{1}{q} \sum_{k\pmod q} \psi(k\overline{n_1}/q) \cdots \psi (k\overline{n_{\ell}}/q) = {\mathcal B}(n_1,\ldots, n_\ell) + O\Big( \frac{\ell K}{q} \log \Big(\frac{eq}{K}\Big) \Big). 
$$ 
\end{proposition} 
\begin{proof} Take $N= \lfloor q/(\ell K)\rfloor$.  The identity \eqref{identity} gives 
\begin{align*}
\sum_{k\pmod q} |  \psi(k\overline{n_1}/q) \cdots \psi (k\overline{n_{\ell}}/q) &-\psi_N(k\overline{n_1}/q) \cdots \psi_N(k\overline{n_\ell}/q)| \\
&\le \frac{1}{2^{\ell-1}} \sum_{j=1}^{\ell} \sum_{k\pmod q} |\psi(k\overline{n_j}/q) - \psi_N(k\overline{n_j}/q)|.
\end{align*}
Using now \eqref{Fex4}, the above is 
\begin{equation} 
\label{propdiscrete1} 
\ll \frac{1}{2^{\ell}} \sum_{j=1}^{\ell} \sum_{k\pmod q} \min \Big( 1 ,\frac{1}{N \Vert k\overline{n_j}/q\Vert} \Big) 
\ll \frac{q}{N} \log (eN). 
\end{equation}  

By Parseval  
\begin{equation} 
\label{propdiscrete2} 
\frac{1}{q} \sum_{k \pmod q} \psi_N(k\overline{n_1}/q) \cdots \psi_N(k\overline{n_\ell}/q) = \Big(\frac{i}{2\pi}\Big)^{\ell} \sum_{\substack{0<|k_j|\le N \\ \sum_j k_j \overline{n_j} \equiv 0 \pmod q}} \frac{1}{k_1\cdots k_\ell} \prod_{j=1}^{\ell} \Big(1- \frac{|k_j|}{N+1}\Big), 
\end{equation} 
which bears a striking resemblance to \eqref{Pars1}.  With our choice for $N$, we claim that in fact the right side of \eqref{propdiscrete2} is exactly equal to the 
expression in \eqref{Pars1}.   Multiplying through by $n_1\cdots n_\ell$, the congruence $\sum k_j \overline{n_j} \equiv 0 \pmod q$ becomes $\sum_j k_j (n_1\cdots n_\ell/n_j) 
\equiv 0\pmod q$.  Since $|k_j| < q/(\ell K)$ and $(n_1 \cdots n_\ell/n_j ) \le K$ for all $j$, it follows that $|\sum_{j} k_j (n_1\cdots n_\ell/n_j) | < q$ so that the congruence becomes the equality $\sum_j k_j (n_1\cdots n_\ell/n_j) =0$, which is the same as the criterion $\sum_j k_j/n_j=0$ of \eqref{Pars1}.  
Combining this observation with \eqref{Pars2} and \eqref{propdiscrete1}, our proposition follows.  
\end{proof}

\section{The moments of ${\widehat s}_q(t)$ and $C(k)$} \label{sec:moments}

We now state our main result on computing the moments of ${\widehat s}_q(t)$ and $C(k)$.  

\begin{theorem} \label{thm4.1}  Let $q$ be a prime, and $\ell$ a natural number.  Then, uniformly in the range 
$\ell\le \sqrt{\log q}/\log \log q$,   
\begin{equation} 
\label{4.1}  
\frac{1}{q} \sum_{k\pmod q} C(k)^\ell = M_C(\ell) + O(q^{-1/(20\ell \log \ell)} ), 
\end{equation} 
where 
$$ 
M_C(\ell) = C^\ell \sum_{n_1, \ldots, n_\ell \ge 1} b(n_1) \cdots b(n_\ell) {\mathcal B}(n_1,\ldots, n_\ell). 
$$ 
The quantity $M_C(\ell)$ equals zero for all odd $\ell$, and for even $\ell$ satisfies 
\begin{equation} 
\label{4.2} 
\frac{e^{\gamma}}{2} (\log\ell  - \log \log \ell +O(1)) \le M_C(\ell)^{\frac 1\ell} \le   \frac{e^{\gamma}}{2} \log \ell +  O(1). 
\end{equation}  
\end{theorem}

\begin{theorem} \label{thm4.2} Let $q$ be a prime, and $\ell$ a natural number.  Then, uniformly in $\ell$,   
\begin{equation*} 
\frac{1}{q} \sum_{t \pmod q} (\pi i {\widehat s}_q(t))^\ell = M_s(\ell) + O(q^{-1/(20\ell\log\ell)} ), 
\end{equation*} 
where 
$$ 
M_s(\ell) = \sum_{n_1, \ldots, n_\ell \ge 1} \frac{ {\mathcal B}(n_1,\ldots, n_\ell)}{n_1\cdots n_\ell} . 
$$ 
The quantity $M_s(\ell)$ equals zero for all odd $\ell$, and for even $\ell$ satisfies 
\begin{equation*} 
\frac{e^{\gamma}}{2} (\log \ell  - \log \log \ell +O(1)) \le M_s(\ell)^{\frac 1\ell} \le  \frac{e^{\gamma}}{2} \log \ell +  O(1). 
\end{equation*}  
\end{theorem} 

We confine ourselves to proving Theorem \ref{thm4.1}, and the proof of Theorem \ref{thm4.2} follows along similar lines.  
In the rest of this section, we establish the asymptotic \eqref{4.1} and the upper bound in \eqref{4.2}; the lower bound in 
\eqref{4.2} needs more work, and will be treated in the next section.  

\begin{proof}[Proof of \eqref{4.1}]   Since $C(-k) = -C(k)$ the odd moments of $C(k)$ vanish.  When $\ell$ is odd, 
${\mathcal B}(n_1,\ldots, n_\ell) =0$ and so the quantity $M_C(\ell)$ is also zero here.   In what follows, we may therefore assume that $\ell$ is an even natural number.

Let $1\le B\le q$ be a parameter to  be chosen shortly.  Note that
\begin{align*}
&\Big| C(k)^\ell - \Big( -C\sum_{n\le B} b(n)\psi\Big(\frac{k\overline{2n}}q\Big)\Big)^\ell \Big|\\
  \le & \Big| C(k) +C  \sum_{n\le B} b(n) \psi\Big(\frac{k\overline{2n}}q\Big)\Big| \cdot \sum_{j=0}^{\ell-1} |C(k)|^j \Big|C \sum_{n\le B} 
b(n) \psi\Big(\frac{k\overline{2n}}q\Big)\Big|^{\ell-1-j} \\
\le &(C_0 \log q)^{\ell-1} \Big| C(k)  + C \sum_{n\le B} b(n) \psi\Big(\frac{k\overline{2n}}q\Big) \Big| ,
\end{align*} 
for some absolute constant $C_0$.  
  By Cauchy-Schwarz and Lemma \ref{lem2.3},  
$$
\frac{1}{q} \sum_{k\pmod q} \Big| C(k) +C \sum_{n\le B} b(n) \psi\Big(\frac{k\overline{2n}}q\Big) \Big| \ll B^{-\frac 12+\epsilon}.
$$ 
We choose $B= q^{1/\ell}$, and (in the range $\ell \le \sqrt{\log q}/\log \log q$) deduce that 
%
\begin{equation*} 
\frac{1}{q} \sum_{k\pmod q} C(k)^\ell = \frac{1}{q} \sum_{k\pmod q} \Big(C\sum_{n\le B} b(n) \psi\Big(\frac{k\overline{2n}}{q} \Big) \Big)^\ell  + O(q^{-\frac{1}{4\ell}}). 
\end{equation*}
Expand out the main term above, replace $k \pmod q$ by $2k \pmod q$, and appeal to Proposition  \ref{propdiscrete} with $K$ 
there being $\le q^{(\ell-1)/\ell}$.   It follows that 
\begin{equation} 
\label{4.3} 
\frac{1}{q} \sum_{k\pmod q} C(k)^\ell = C^{\ell} \sum_{n_1,\ldots, n_\ell \le q^{1/\ell} } b(n_1)\cdots b(n_\ell) {\mathcal B}(n_1,\ldots,n_\ell) 
+ O( q^{-\frac{1}{4\ell}}). 
\end{equation} 

It remains now to bound the difference between the main term in \eqref{4.3} and the expression for $M_C(\ell)$, which is 
$$ 
\le C^{\ell} \sum_{n > q^{1/\ell}} \sum_{n_1\cdots n_\ell = n} b(n_1) \cdots b(n_\ell) | {\mathcal B}(n_1,\ldots, n_\ell)| \le 
(C/2)^{\ell} \sum_{n > q^{1/\ell}} \sum_{n_1\cdots n_\ell = n} b(n_1) \cdots b(n_\ell) \frac{1}{\mathrm{sf}(n)}, 
$$ 
where $\mathrm{sf}(n)$ is the largest squarefree divisor $d$ of $n$ that is coprime to $n/d$.  
We estimate the sum above by Rankin's trick; with $\alpha= 1/(10\log \ell)$ the above is 
\begin{align*}
&\le (C/2)^{\ell} q^{-\alpha/\ell} \sum_{n=1}^{\infty} \sum_{n_1\cdots n_\ell = n} b(n_1) \cdots b(n_\ell) \frac{n^{\alpha}}{\mathrm{sf}(n)} 
\\
&\le e^{O(\ell)} q^{-\alpha/\ell} \prod_{p\ge 3} \Big( 1+ \frac{\ell p^{\alpha}}{p(p-2)} + \sum_{j=2}^{\ell} \binom{\ell}{j} \frac{p^{j\alpha}}{(p-2)^j}
\Big),
\end{align*} 
upon recalling the definition of $b(n)$.   The contribution of primes $p\le \ell$ to the product above is 
$$ 
\le \prod_{3\le p\le \ell} \Big(1 +\frac{p^{\alpha}}{p-2} \Big)^\ell \le (\log \ell)^{\ell} e^{O(\ell)},
$$ 
while the contribution of primes $p>\ell$ to the product above is 
$$ 
\ll \prod_{p> \ell} \exp\Big( O\Big(\frac{\ell^2 p^{2\alpha}}{p^2} \Big)  \Big) = e^{O(\ell)}. 
$$ 
We conclude that the difference between the main term in \eqref{4.3} and the expression for $M_C(\ell)$ is 
$$
\ll (\log \ell)^{\ell} e^{O(\ell)} q^{-\alpha/\ell}  \ll q^{-1/(20\ell \log \ell)}, 
$$ 
completing the proof of \eqref{4.1}. 
\end{proof}

\begin{proof}[Proof of the upper bound in \eqref{4.2}]  Note that 
\begin{align*}  
M_C(\ell) &\le  C^\ell \sum_{n_1, \ldots, n_\ell} b(n_1)\cdots b(n_{\ell})|{\mathcal B}(n_1,\ldots, n_\ell)| 
\le (C/2)^{\ell} \sum_{n_1,\ldots, n_{\ell}} \frac{b(n_1)\cdots b(n_\ell)}{\mathrm{sf}(n)} \\
&\le (C/2)^{\ell} \prod_{p\ge 3} \Big( 1 +\frac{\ell}{p(p-2)} + \sum_{j=2}^{\ell} \binom{\ell}{j} \frac{1}{(p-2)^j} \Big). 
\end{align*}
The contribution of primes $p\le \ell$ is 
$$ 
\le \prod_{3 \le p \le \ell} \Big( 1+ \frac{1}{p-2} \Big)^{\ell} = \Big( \prod_{3 \le p \le \ell} \Big( 1 -\frac{1}{(p-1)^2}\Big)^{-1} \Big(1-\frac{1} {p}\Big)^{-1}\Big)^{\ell}  = C^{-\ell} (e^{\gamma} \log \ell + O(1) )^{\ell}, 
$$ 
upon using Mertens's theorem.   The contribution of primes $p> \ell$ is 
$$ 
\exp \Big( \sum_{p > \ell} O \Big(\frac{\ell^2}{p^2} \Big) \Big) = \exp \Big( O\Big( \frac{\ell}{\log \ell}\Big)\Big),
$$ 
and so the upper bound in \eqref{4.2} follows.   
\end{proof}

\section{Completing the proof of Theorem \ref{thm4.1}: Proof of the lower bound  in \eqref{4.2}} 

To obtain the lower bound in \eqref{4.2} we take an indirect approach, working with a 
continuous model that has the same moments as $C(k)$.   Let $B$ be a positive integer, and let 
$L(B)$ denote the least common multiple of the natural numbers $n\le B$.   For a real number $x$, define 
$$ 
C(x;B) = C \sum_{n\le B} b(n) \psi(x/n).    
$$ 
It follows readily that 
$$ 
\frac{1}{L(B)} \int_0^{L(B)} C(x;B)^\ell dx  = C^{\ell} \sum_{n_1, \ldots, n_{\ell} \le B} b(n_1) \cdots b(n_\ell) {\mathcal B}(n_1,\ldots, n_{\ell} ), 
$$ 
so that 
\begin{equation} 
\label{5.1} 
M_{C}(\ell) = \lim_{B\to \infty} \frac{1}{L(B)} \int_0^{L(B)} C(x;B)^{\ell} dx. 
\end{equation}  
We shall obtain a lower bound for the right side of \eqref{5.1}; naturally, we may assume that $\ell$ is even and large.  

Suppose that $B>\ell$, and put $\ell_0 = \ell/\log \ell$.   
Let ${\mathcal I}$ denote the subset of $[0,L(B)]$ consisting of points $x= k L(\ell_0) -y$ with $1\le k \le L(B)/L(\ell_0)$, and $0< y\le 1/10$.  
Let $\psi^+(t) = \psi(t)$ whenever $t$ is not an integer, and $\psi^+(t) = 1/2$ when $t$ is an integer.  Then for $x=k L(\ell_0)-y\in {\mathcal I}$ 
note that 
$$
C(x;B) = C \sum_{n\le B}  b(n) \psi((k L(\ell_0) -y)/n) = C\sum_{n \le B} b(n) \Big( \psi^+ (kL(\ell_0)/n) - y/n\Big). 
$$ 
Since, for $n\le B$, 
$$ 
\sum_{k=1}^{L(B)/L(\ell_0)} \psi^{+}\Big( \frac{kL(\ell_0)}{n} \Big) = \frac 12 \frac{L(B)}{L(\ell_0)} \frac{(n,L(\ell_0))}{n}, 
$$ 
it follows that (note $|{\mathcal I}| = L(B)/(10L(\ell_0))$)
$$ 
\frac{1}{|{\mathcal I}|} \int_{\mathcal I} C(x;B) dx = \frac{C}{2} \sum_{n\le B} b(n) \frac{(n,L(\ell_0))-1/20}{n}, 
$$  
and therefore by H{\" o}lder's inequality that 
$$ 
\frac{1}{L(B)} \int_0^B C(x,B)^\ell dx \ge \frac{1}{10L(\ell_0)} \frac{1}{|{\mathcal I}|} \int_{\mathcal I} C(x;B)^{\ell}dx \ge 
\frac{1}{10L(\ell_0)} \Big(\frac C2 \sum_{n\le B} b(n) \frac{(n,L(\ell_0))-1/20}{n}\Big)^{\ell}. 
$$ 
Now letting $B\to \infty$, we find by \eqref{5.1} that 
$$ 
M_C(\ell) \ge \frac{1}{10 L(\ell_0)} \Big( \frac C2\sum_{n=1}^{\infty} \frac{b(n) (n,L(\ell_0))}{n} +O(1)\Big)^{\ell} \ge e^{-O(\ell_0)} \Big(\frac C2 \prod_{3\le p\le \ell_0} 
\Big(1 + \frac{1}{p-2}\Big) + O(1) \Big)^{\ell},  
$$ 
upon using the prime number theorem to estimate $L(\ell_0)$, and recalling the definition of $b$.  
Now 
$$ 
\frac C2 \prod_{3\le p\le \ell_0} \Big(1+\frac{1}{p-2}\Big) = \prod_{3 \le p\le \ell_0} \Big(1-\frac 1p\Big)^{-1} \Big( 1 + O\Big( \frac{1}{\ell_0}\Big)\Big) = 
\frac{e^{\gamma}}{2} \log \ell_0 +O(1), 
$$ 
and therefore the lower bound  in \eqref{4.2} follows. 

\section{Proof of Theorem \ref{thm1}}  \label{sec:proofthm}

\begin{proof}[Proof of Part 1]  Theorem \ref{thm4.1} shows that all the moments of $C(k)$ exist, and do not 
grow too rapidly.   The moment generating function $\sum_{\ell =0}^{\infty} x^\ell M_{C}(\ell)/\ell!$ converges for 
all $x$, and therefore the sequence of moments $M_C(\ell)$ uniquely determines a distribution, which is the 
limiting distribution for $C(k)$.   Since $C(k) = -C(-k)$, the limiting distribution is clearly symmetric around $0$.   

To gain an understanding of this limiting distribution, and to establish its continuity, it is helpful to think of the 
continuous model $C(x;B)$ discussed in Section 5.  Consider the characteristic function (that is, Fourier transform) of 
$C(x;B)$; namely 
$$ 
{\Bbb E}( e^{it C(x,B)})  = \frac{1}{L(B)} \int_0^{L(B)} e^{itC(x,B)} dx. 
$$ 
Omit the measure zero set of integers $x$, and write $x=k-y$ with $1\le k\le L(B)$ and $0< y < 1$.  Then, with $\psi^+$ as 
in Section 5 and $C^+(x;B) = C\sum_{b\leq B} b(n) \psi^+(x/b)$, we have $C(x;B) = C^+(k;B) - y \sum_{n\le B} b(n)/n$, and so 
\begin{equation} 
\label{6.0} 
 \frac{1}{L(B)} \int_0^{L(B)} e^{itC(x,B)} dx
 =\frac{1}{L(B)} \sum_{k=1}^{L(B)} e^{it C^+(k,B)} \int_{0}^{1} e^{-ity \sum_{n\le B} b(n)/n} dy \ll \frac{1}{1+|t|}. 
 \end{equation} 

 Given an interval $I = (\alpha-\epsilon, \alpha+\epsilon)$ with $\epsilon <1/2$, we can readily find a majorant $\Psi(x)$ of 
 the indicator function of $I$, with $|{\widehat \Psi}(x)| \ll \epsilon/ (1+(\epsilon x)^2)$.  For example take $\Psi(x) = 
 \max(2 - |x-\alpha|/\epsilon, 0)$, which is a relative of the Fejer kernel.   
 Then by Fourier inversion 
 \begin{align*}
 \frac{1}{L(B)} \int_{\substack{x\in [0, L(B)] \\ C(x,B) \in I } } dx &\le 
 \frac{1}{L(B)} \int_{0}^{L(B)} \Psi(C(x,B)) dx \\
 &= \int_{-\infty}^{\infty} {\widehat \Psi}(t) {\Bbb E}(e^{it C(x,B)}) dt 
 \ll \int_{-\infty}^{\infty} \frac{1}{1+|t|} \frac{\epsilon}{1+(\epsilon t)^2}  dt  \ll \epsilon  \log (1/\epsilon). 
 \end{align*} 
Therefore $C(x,B)$ has a continuous distribution, and the continuity is uniform in $B$, so that letting $B\to \infty$, we 
conclude that the limiting distribution for $C(k)$ is also continuous.  
\end{proof}

\begin{proof}[Proof of Parts 2 and 3]  Since Part 3 follows upon taking $x =( \frac 12-\epsilon)\log \log q$ in Part 2, 
it is enough to prove Part 2.  For any even $\ell \le \sqrt{\log q}/\log \log q$, we see using Theorem \ref{thm4.1} that 
$$ 
\frac{1}{q} \# \{ k\pmod q: C(k) \ge \frac{e^{\gamma}}{2} x \} \,\,\le\,\, \Big(\frac{e^{\gamma}}{2} x\Big)^{-\ell} (M_{C}(\ell) +o(1)) 
\,\,\ll\,\, \Big( \frac{\log \ell +O(1)}{x} \Big)^\ell. 
$$ 
Choosing $\ell$ to be an even integer around $A e^{x}$ for a suitably small positive constant $A$, the upper bound in Part 2 
follows.  

To establish the lower bound in Part 2, note that for even $\ell \le \sqrt{\log q}/(2\log \log q)$, we have by Theorem \ref{thm4.1} 
\begin{equation} 
\label{6.1} 
\Big( \frac{e^{\gamma}}{2} (\log \ell -\log \log \ell + O(1))\Big)^{\ell} \ll \frac 1q \sum_{k \pmod q} C(k)^\ell.  
\end{equation} 
The contribution from terms $k$ with $|C(k)| \le \frac{e^{\gamma}}{2} (\log \ell -\log \log \ell -A)$ for a suitably large constant $A$ 
is clearly negligible compared to the right side of \eqref{6.1}.   The contribution from terms $k$ with $|C(k)| 
\ge \frac{e^{\gamma}}{2} (\log \ell +\log \log \ell +A)$ for a suitably large constant $A$ is 
\begin{align*}
&\le \Big( \frac{e^{\gamma}}{2} (\log \ell +\log \log \ell +A)\Big)^{-\ell} \frac{1}{q} \sum_{k \pmod q} C(k)^{2\ell}  
\\
&\ll \Big( \frac{e^{\gamma}}{2} (\log \ell +\log \log \ell +A)\Big)^{-\ell} \Big(\frac{e^{\gamma}}{2} \log \ell +O(1) \Big)^{2\ell}, 
\end{align*} 
upon using Theorem \ref{thm4.1} to estimate the $2\ell$-th moment.   If $A$ is suitably large, then this too is negligible in comparison 
to the right side of \eqref{6.1}.   Therefore it is the terms with $|C(k)|$ lying between $ \frac{e^{\gamma}}{2} (\log \ell -\log \log \ell -A)$ 
and $ \frac{e^{\gamma}}{2} (\log \ell +\log \log \ell +A)$ that account for the bulk of the contribution to \eqref{6.1}, and 
so 
\begin{align*}
&\Big( \frac{e^{\gamma}}{2} (\log \ell +\log \log \ell +A)\Big)^\ell \,\,\frac{1}{q} \# \{ k: |C(k)| \ge \tfrac{e^{\gamma}}{2} (\log \ell -\log \log \ell -A)\}\\ 
\gg &\Big( \frac{e^{\gamma}}{2} (\log \ell -\log \log \ell + O(1))\Big)^\ell.
\end{align*}
Choosing $\ell$ of size $xe^{x}$, the lower bound in Part 2 follows.   
\end{proof} 

\begin{proof}[Proof of Part 4]   First suppose that $C(k)$ is negative. From \cite{Montgomery10} (Chapter 1, page 6) we recall that for each natural number $K$ there 
is a trigonometric polynomial 
$$ 
B_K(x) = \frac{1}{2(K+1)} + \sum_{1\le |j|\le K} c_j e(jx) 
$$ 
with $c_j \ll 1/j$, such that $B_K(x) \ge \psi(x)$ for all $x$.    Using Lemma \ref{lem2.2} with $N=q^8$ we obtain 
$$ 
0 \le  -C(k) = C\sum_{\substack{n\le q^8 \\ (n,q)=1}} b(n) \psi(k\overline{2n}/q) + O(1) \le C \sum_{n\le q^{8}} b(n) B_K(k\overline{2n}/q) + O(1).
$$ 
Thus, for some positive constant $A$, 
\begin{equation} 
\label{6.2} 
-C(k) \le A \Big(1+ \frac{1}{K+1} \sum_{n\le q^{8}} b(n) + \sum_{1\le |j| \le K} \frac 1j \Big| \sum_{\substack{ n\le q^{8}\\ (n,q)=1}} b(n) e\Big( \frac{kj \overline{2n}}{q}\Big)\Big|\Big). 
\end{equation}  
At this stage, we need the following result which follows from work of Bourgain and Garaev \cite{BG} (refining earlier work of Karatsuba \cite{Ka}; see also Korolev \cite{Ko}).  

\begin{lemma} \label{lem6.1}   Let $q$ be a prime, and $a$ be any integer coprime to $q$.  Then for all $N\ge 1$ 
$$ 
\Big| \sum_{\substack{n\le N \\ (n,q)=1}} \frac 1n e\Big(\frac{a\overline{n}}{q}\Big) \Big| \ll (\log q)^{\frac 23} (\log \log q)^2. 
$$
\end{lemma}
\begin{proof}   Theorem 16 of Bourgain and Garaev \cite{BG} gives 
$$ 
\Big| \sum_{n\le x} e\Big( \frac{a\overline{n}}{q}\Big) \Big| \ll \frac{x}{(\log x)^{\frac 32}} \log q (\log \log q)^3. 
$$ 
Partial summation using this bound for $x\ge \exp((\log q)^{\frac 23} (\log \log q)^2)$, and the trivial bound (that the sum is at most $x$) for smaller $x$ yields 
the lemma.
\end{proof}  

Returning to \eqref{6.2}, take there $K=\lfloor \log q\rfloor$.   Then the right side of \eqref{6.2} is (recalling the definition $b(n) = \sum_{uv=n} a(u)/v$) 
$$ 
\ll 1 + \sum_{j\le K} \frac 1j \sum_{\substack{ u\le q^8  \\ (u,q)=1}} |a(u)| \Big| \sum_{\substack{v\le q^{8}/u \\ (v,q)=1}} \frac 1v e \Big( 
\frac{kj\overline{2uv}}{q}\Big) \Big| \ll (\log q)^{\frac 23} (\log \log q)^3,
$$
using Lemma \ref{lem6.1} and since $\sum_{n} |a(n)| \ll 1$.  This proves that $-C(k) \le A (\log q)^{\frac 23} (\log \log q)^3$, which is the 
desired bound in the case $C(k)$ negative.   Arguing similarly with a minorant for $\psi(x)$ instead of a majorant, leads to the 
same bound for $C(k)$ in the case when it is positive.
\end{proof} 

\begin{proof}[Proof of Part 5] Applying Lemma \ref{lem2.3} we find that 
$$ 
\frac{1}{q} \sum_{k \pmod q} |C(k) -C(k+m)|^2 \ll B^{-1+\epsilon} + \frac{1}{q} \sum_{k\pmod q} \Big| \sum_{n\le B} b(n) \Big(\psi\Big(\frac{(k+m)\overline{2n}}{q} 
\Big) - \psi\Big(\frac{k\overline{2n}}{q}\Big)\Big)\Big|^2. 
$$ 
Using Cauchy-Schwarz the second term above is 
\begin{equation} 
\label{6.4}
\ll \frac 1q \Big(\sum_{n\le B} b(n) \Big) \sum_{n\le B} b(n) \sum_{k\pmod q} \Big(\psi\Big( \frac{k + m\overline{2n}}{q} \Big) -\psi\Big(\frac{k}{q}\Big)\Big)^2,
\end{equation} 
where in the inner sum we replaced $k$ by $2kn$.  Since $|\psi((k+a)/q) - \psi(k/q)|\le |a|/q$ unless there is an integer between $k/q$ and $(k+a)/q$, 
we may check that 
$$ 
\frac 1q \sum_{k \pmod q} \Big( \psi \Big( \frac{k+a}{q}\Big) - \psi\Big( \frac kq\Big) \Big)^2 \ll \frac{a}{q}. 
$$
Since $m$ is a multiple of all numbers $B$ (and recalling that $b(n)=0$ unless $n$ is odd), we may write $m\overline{2n} = qr +a$ 
with $a=m/(2n)$.  Therefore the quantity in \eqref{6.4} is 
$$ 
\ll (\log B) \sum_{n\le B} \frac{m}{nq} \ll \frac{m}{q} \log B, 
$$ 
completing our proof.
\end{proof} 

\section{Proof of Theorem \ref{thm3}} \label{sec:montgomery}

As in the proofs of Theorems \ref{thm1} and \ref{thm2}, the main result is to compute the moments of $\widetilde{R}(u)$.  
The proof of Theorem \ref{thm3} then follows in exactly the same way as the corresponding parts of Theorem \ref{thm1}.

\begin{theorem}\label{thm:R-moments}  There is a positive number $c<1$ such that 
  uniformly for all natural numbers $\ell$ in the range $\ell \leq \frac{c}{9}{\sqrt{\log y}}/{\log\log y}$, we have
\[
\frac{1}{y} \int_0^y \widetilde{R}(u)^\ell \, du = M_R(\ell) + O\Big(\exp\Big(-\frac{c}{8}\sqrt{\log y}\Big)\Big),
\]
where
\[
M_R(\ell) = \sum_{n_1,\dots,n_\ell} \frac{\mu(n_1)\dots\mu(n_\ell)}{n_1\dots n_\ell} \mathcal{B}(n_1,\dots,n_\ell).
\]
For odd $\ell$, $M_R(\ell)=0$, while $M_R(2) = 1/2\pi^2$ and for even $\ell \geq 4$ we have
\[
M_R(\ell) \leq \Big(\frac{3e^\gamma}{\pi^2} \log \ell + O(1) \Big)^{\ell}.
\] 
\end{theorem}

We begin with a lemma, which will allow us to truncate ${\widetilde R}(u)$ by a short sum of sawtooth functions.  

\begin{lemma}\label{lem:mont}   For all $1\le N\le y$ we have 
$$ 
\sum_{N < n_1, n_2 \le 2N } \Big| \frac 1y \int_0^y \psi(x/n_1) \psi(x/n_2) dx \Big| \ll (\log y)^2 \Big( N + N^2 \frac{\sqrt{N}}{\sqrt{y}}\Big). 
$$ 
 \end{lemma}
\begin{proof}  Let $K \geq 2$ be a parameter to be chosen shortly, and let $\psi_K(x)$ be as in \eqref{Fex3}.  First note that 
$$ 
\frac 1y \int_0^y |\psi(x/n_1)\psi(x/n_2) - \psi_K(x/n_1)\psi_K(x/n_2)| dx 
\le \frac 1y \int_0^y \sum_{j=1}^{2} |\psi(x/n_j)-\psi_K(x/n_j)|  dx \ll \frac{1}{K}, 
$$ 
upon using \eqref{Fex5}, and since $n_1$ and $n_2$ are at most $N\le y$.   
Next, from the Fourier expansion of $\psi_K$ (see \eqref{Fex3}) it follows 
that 
\begin{align*}
\frac{1}{y}\Big| \int_0^y \psi_K(x/n_1) \psi_K(x/n_2) dx \Big| 
&\ll \sum_{0< |k_1|, |k_2| \le K} \frac{1}{|k_1k_2|} \Big|\frac 1y \int_0^y e\Big(x \Big(\frac{k_1}{n_1} + \frac{k_2}{n_2}\Big) \Big) dx \Big| 
\\
&\ll \sum_{0< |k_1|, |k_2| \le K} \frac{1}{|k_1k_2|} \min \Big( 1, \frac{1}{y|k_1/n_1+k_2/n_2|}\Big). 
\end{align*} 
From these two estimates it follows that the sum to be bounded is 
$$ 
\ll \frac{N^2}{K} + \sum_{0< |k_1|, |k_2| \le K} \frac{1}{|k_1 k_2|} \sum_{N < n_1, n_2 \le 2N} \min \Big( 1, \frac{1}{y|k_1/n_1 + k_2/n_2|}\Big). 
$$ 

To estimate the sum above, we split the terms into two groups: those with $|k_1/n_1+k_2/n_2| \ge K/y$ and those terms with 
$|k_1/n_1+k_2/n_2| < K/y$.  The first group contributes 
$$ 
\ll \frac {N^2}K \sum_{0< |k_1|, |k_2| \le K} \frac{1}{|k_1 k_2|}  \frac{1}{|k_1 k_2|} \ll \frac{N^2}{K} (\log K)^2. 
$$ 
Terms in the second group only exist for $k_1$ and $k_2$ of opposite sign, and here $|k_1n_2 + k_2 n_1| \ll KN^2/y$, so that 
if $k_1$, $n_1$, and $k_2$ are fixed, then $n_2$ has $\ll 1 + KN^2/y$ choices.  Therefore the second group contributes 
$$
\ll \Big(1 + \frac{KN^2}{y}\Big) N \sum_{0< |k_1|, |k_2| \le K} \frac{1}{|k_1 k_2|} \ll  (\log K)^2 N \Big(1+ \frac{KN^2}{y}\Big). 
$$ 
Choosing $K= 2\lceil \sqrt{y/N} \rceil$, the lemma follows.  
\end{proof}

\begin{proof}[Proof of Theorem \ref{thm:R-moments}]   From Theorem 1 and Lemma 1 of \cite{Montgomery} (but beware of the 
changes in notation, especially that his saw tooth function differs from ours in sign)  it follows that with $N= y\exp(-c\sqrt{\log y})$ 
for a suitable positive constant $c<1$, one has 
$$ 
{\widetilde R}(u) = -\sum_{n\le N} \frac{\mu(n)}{n} \psi(u/n)  + O(\exp(-c\sqrt{\log y})), 
$$ 
for all $N\le u \le y$.   Since ${\widetilde R}(u)$ and the sum over $n$ above are $\ll \log y$, it follows that for $\ell \le 
\frac c9 \sqrt{\log y}/\log \log y$ 
\begin{equation} 
\label{7.1} 
\frac 1y \int_0^y {\widetilde R}(u)^{\ell} du = \frac{(-1)^{\ell}}{y} \int_0^y \Big( \sum_{n\le N} \frac{\mu(n)}{n}\psi(u/n)\Big)^{\ell} du + 
O(\exp(-\tfrac c2 \sqrt{\log y} ) ). 
\end{equation} 
Now applying \eqref{identity} we see that 
\begin{align}  
\label{7.2} 
\frac{1}{y} \int_0^y  \Big( \sum_{n\le N} \frac{\mu(n)}{n}\psi(u/n)\Big)^{\ell} du 
&= \frac 1y \int_0^y  \Big( \sum_{n\le y^{1/(2\ell)}} \frac{\mu(n)}{n}\psi(u/n)\Big)^{\ell} du \nonumber\\ 
&+ 
O \Big( \frac{\ell (\log y)^{\ell-1}}{y} \int_0^y \Big| \sum_{y^{1/(2\ell)} \le n \le N} \frac{\mu(n)}{n} \psi(u/n) \Big| du \Big). 
\end{align} 
 
 Expanding out, the main term in \eqref{7.2} is 
\begin{align*}
& \sum_{n_1, \ldots, n_\ell \le y^{1/(2\ell)}} \frac{\mu(n_1) \cdots \mu(n_\ell)}{n_1\cdots n_\ell} \frac 1y \int_{0}^y \prod_{j=1}^{\ell} \psi(u/n_j) du \\
 = & \sum_{n_1, \ldots, n_\ell \le y^{1/(2\ell)}} \frac{\mu(n_1) \cdots \mu(n_\ell)}{n_1\cdots n_\ell} ({\mathcal B}(n_1,\ldots, n_\ell) + O(n_1\cdots n_\ell)). 
 \end{align*}
 Arguing as in the proof of Theorem \ref{thm4.1}, this may be seen to equal $M_R(\ell) + O(y^{-1/(40\ell \log \ell)})$. 
 
 As for the remainder term in \eqref{7.2}, splitting the terms $y^{1/(2\ell)} \le n\le N$ into dyadic blocks, we 
 may bound this by 
 $$ 
 \ll \exp(\tfrac c8\sqrt{\log y}) \max_{\substack{ y^{1/(2\ell)} \le M \le N \\ I \subset [M,2M]} } 
 \frac{1}{y} \int_0^y \Big| \sum_{n\in I} \frac{\mu(n)}{n} \psi(u/n) \Big| du, 
 $$ 
 where the maximum is over subintervals $I$ of $[M,2M]$.   By Cauchy-Schwarz and Lemma \ref{lem:mont}, this is 
 $$ 
 \ll  \exp(\tfrac c8\sqrt{\log y}) \max_{\substack{ y^{1/(2\ell)} \le M \le N \\ I \subset [M,2M]} } (\log y) \Big( \frac{1}{M} + \frac{\sqrt{M}}{\sqrt{y}}\Big)^{\frac 12} 
 \ll \exp(-\tfrac c8 \sqrt{\log y}). 
 $$ 
This justifies the first claim of the theorem.  It is also clear that $M_R(\ell) =0$ for odd $\ell$, and the formula for $M_R(2)$ follows from 
our knowledge of ${\mathcal B}(n_1,n_2)$.  Lastly, the claimed upper bound on $M_R(\ell)$  follows exactly as the upper bound for $M_C(\ell)$ in Theorem \ref{thm4.1}. 
 \end{proof}
 
\bibliographystyle{abbrv}
\bibliography{distribution}

\end{document}